\documentclass[12pt]{article}
\usepackage{etex}
\usepackage{pifont}
\usepackage{soul}
\usepackage[dvips]{graphicx}
\usepackage{amsmath,latexsym,amsthm,cmap,indentfirst,setspace,ccaption,amssymb,color,amscd,mathrsfs,euscript}

\usepackage{newpxtext} 
\usepackage{mathpazo} \usepackage[euler-digits]{eulervm}

\usepackage[normalem]{ulem}
\usepackage[svgnames]{xcolor}

\newcommand{\nicecolor}{Navy}
\usepackage[pdfauthor={P. Autissier, J.-P. Furter \& E. Yasinsky}, backref=page, colorlinks, linktocpage, citecolor=\nicecolor, linkcolor=\nicecolor, urlcolor=\nicecolor]{hyperref}
\usepackage{array}

\usepackage{comment}
\usepackage{enumitem}
\usepackage{amssymb,lipsum}
\usepackage{float,xypic}
\usepackage{cmap,longtable,booktabs}
\usepackage[all,cmtip]{xy}

\usepackage{geometry}
\usepackage{caption, subcaption}

\usepackage{verbatim}
\hypersetup{
    bookmarks=true,         
    unicode=true,          
    pdftoolbar=true,        
    pdfmenubar=true,        
    pdffitwindow=false,     
    pdfnewwindow=true,      
    colorlinks=true,       
    linkcolor=blue,          
    citecolor=blue,        
    filecolor=magenta,      
    urlcolor=blue          
}

\newcommand\A{{\mathbb A}}

\DeclareMathOperator{\Iterates}{\mathcal I}

\DeclareMathAlphabet{\mathpzc}{OT1}{pzc}{m}{it}

\DeclareMathOperator{\End}{End}
\DeclareMathOperator{\Bir}{Bir}

\DeclareMathOperator{\Aut}{Aut}

\DeclareMathOperator{\SL}{SL}

\theoremstyle{plain}
\newtheorem{theorem}{Theorem}[section]
\newtheorem{lemma}[theorem]{Lemma}

\newtheorem{proposition}[theorem]{Proposition}

\newtheorem*{maintheorem}{Theorem}

\newtheorem*{keyproposition}{Key proposition}

\theoremstyle{definition}
\newtheorem{definition}[theorem]{Definition}
\newtheorem{remark}[theorem]{Remark}
\newtheorem{question}[theorem]{Question}

\theoremstyle{definition}

\title{Iterated polynomials are dense}

\author{Pascal Autissier \and Jean-Philippe Furter \and Egor Yasinsky}

\date{}

\newcommand{\CC}{\mathbb C}

\newcommand{\FF}{\mathbf F}
\newcommand{\id}{\mathrm{id}}
\newcommand{\NN}{\mathbb N}
\newcommand{\RR}{\mathbb R}

\newcommand{\kk}{\mathbf{k}}

\renewcommand{\AA}{\mathbb A}
\newcommand{\PP}{\mathbb P}
\newcommand{\ZZ}{\mathbb Z}

\newcommand{\wordmap}{\mathbf w}

\newcommand{\address}[1]{\par\smallskip\noindent\textit{#1}}
\newcommand{\email}[1]{\par\smallskip\noindent\texttt{#1}}

\begin{document}

\maketitle

\begin{abstract}
	For any infinite field $\kk$ and any positive integer 
	$r$, we show constructively that the map sending each polynomial $P\in\kk[x]$ to its $r$-th iterate $P^{\circ r}$ is dominant in various inductive limit topologies on the space of all polynomials.	
\end{abstract}

\tableofcontents

\section{Equations in groups}\label{sec:intro}

\subsection{The word map} 
In what follows, by a \emph{word}
we mean an element $w=w(x_1,\ldots,x_N)$ of the free group~ $\mathcal{F}_N$ on $N$ generators $x_1,\ldots,x_N$, for some $N\geq 1$. Let $G$ be a group. The \emph{word map} on $G$ defined by $w$ is the map 
\[
\wordmap\colon G^N\to G,\ \ (g_1,\ldots,g_N)\mapsto w(g_1,\ldots,g_N).
\]
Assuming that $\wordmap\ne\id$, what can be said about the image of $\wordmap$? Or, in other words, what can be said about the solutions $(g_1,\ldots,g_N)\in G^N$ of the equation 
\begin{equation}\label{eq: group equation}
	w(g_1,\ldots,g_N)=g\ \ \ \tag{$*$}
\end{equation}
when one varies $g\in G$? This problem has a long and remarkable history, and we refer to \cite{Shalev,BandmanGarionKunyavskii,GordeevKunyavskiiPlotkin2018} for some excellent surveys. Clearly, one cannot expect the map $\wordmap$ to be surjective already for power words $x^r$, $r\geq 2$. For example, taking $G=\SL_2(\kk)$, where $\kk=\RR$ or $\kk=\CC$, we easily find matrices which do not admit square roots in $G$, and thus $w=x^2$ induces a non-surjective word map. 

\subsection{Borel's dominance theorem and weakly exponential Lie groups} Still, one can ask: how ``large'' is the image $\wordmap(G^N)$ of a word map? For instance, if the group 
$G$ is endowed with a reasonable topology, is the image of $\wordmap$ dense in $G$? Here are two prototypical results in this direction which inspire our main theorem, presented in Section \ref{sec: main result} below.

A morphism $f\colon X\to Y$ of topological spaces is called \emph{dominant} if its image $f(X)$ is dense in $Y$. When $X$ and $Y$ are algebraic varieties, equipped with the Zariski topology, Chevalley's theorem implies that a morphism $f\colon X\to Y$ is dominant if and only if $f(X)$ contains a non-empty Zariski open subset. The following remarkable theorem is due to A.~Borel. 

\begin{theorem}[{\cite{Borel}}]
		If $\kk$ is a field, $G$ is a connected semisimple linear algebraic $\kk$-group, and $w\ne\id$, then the corresponding word map $\wordmap\colon G^N\to G$ is dominant.
\end{theorem}

The second result concerns complex Lie groups and power maps $\wordmap\colon g\mapsto g^r$ defined on them. Standard tools from Lie theory, such as the exponential map $\exp\colon\mathfrak{g}\to G$, provide an easy way to extracting $r$-th roots. Specifically, given an element $g\in G$, assume there exists $v\in\mathfrak{g}$ such that $g=\exp(v)$. Then $\exp(v/r)$ is an $r$-th root of $g$.  Unfortunately, the exponential map of a Lie group is not necessarily surjective --- this issue has been the subject of extensive study (see the references in the surveys mentioned above). However, for complex Lie groups the following can be said\footnote{In the semisimple case, this theorem is also due to A. Borel.}.

\begin{theorem}[{\cite[Theorem 2.11]{HofmannMukherjea}}]
	Every complex connected Lie group $G$ is weakly exponential, i.e. the image of the exponential map $\exp\colon\mathfrak{g}\to G$ is dense.
\end{theorem}

Hence, for $w=x^r$ and ``typical'' $g$, the equation \eqref{eq: group equation} admits a solution.

\subsection{Endomorphisms of the affine space and some ind-topologies} 
This note was motivated by a question posed in the surveys \cite{GordeevKunyavskiiPlotkin2018,BandmanGarionKunyavskii}, which deal with word equations in simple matrix groups. Let $\kk$ be any field. Noting J.-P. Serre's observation \cite{SerreBourbakiCremona} that the group $\Bir(\PP_\kk^2)$ of birational transformations of the projective plane  resembles simple linear algebraic groups, the authors of these surveys asked whether word maps are dominant for the Cremona group $G=\Bir(\PP_\kk^2)$ and any known ``interesting'' topology on it, see \cite[Problem 7.11]{BandmanGarionKunyavskii} and \cite[Question 3.11]{GordeevKunyavskiiPlotkin2018}. 

We consider the closely related automorphism group of the affine space $\Aut(\AA_\kk^n)$ and, more generally, the monoid $\End(\AA_\kk^n)$ of algebraic endomorphisms. Recall that an endomorphism of $\AA^n_\kk$ is given by
\[
f\colon x=(x_1,\ldots,x_n)\mapsto (f_1(x),\ldots,f_n(x)),
\]
where $f_1,\ldots,f_n\in\kk[x_1,\ldots,x_n]$ are polynomials. To simplify the notation, we often write $(f_1,\ldots,f_n)$ in what follows. The \emph{degree} $\deg f$ of an endomorphism $f=(f_1,\ldots,f_n)$ is defined as $\deg(f)=\max\{\deg(f_1),\ldots,\deg(f_n)\}$. The subset of invertible elements of the monoid $\mathcal{E}=\End(\AA^n_\kk)$ is the group $\Aut(\AA^n_\kk)$ of automorphisms of $\AA^n_\kk$.

For each $d\geq 0$, we identify the set 
\[
\mathcal{E}_{\leq d}=\big \{f\in\End(\AA_\kk^n)\colon \deg f\leq d \big\}
\]
with the vector space $\kk^N$, where $N={\binom{d+n}{n} n}$, in an obvious way. Following I. R. \v{S}afarevich \cite{ShafarevichIndGroups}, we view $\End(\AA_\kk^n)$ as an \emph{ind-monoid}.

\begin{definition}\label{definition: topologies}
	Consider the filtration
	\begin{equation}\label{eq: filtration}
		\mathcal{E}_{\leq 1}\subseteq\mathcal{E}_{\leq 2}\subseteq\ldots\subseteq\mathcal{E}_{\leq d}\subseteq\mathcal{E}_{\leq d+1}\subseteq\ldots,\ \ \mathcal{E}=\bigcup_{d=1}^\infty\mathcal{E}_{\leq d}.
	\end{equation}
	\begin{enumerate}
		\item A set $S\subseteq\mathcal{E}$ is called \emph{Zariski closed} if $S\cap\mathcal{E}_{\leq d}$ is Zariski closed in $\mathcal{E}_{\leq d}$ for all $d\geq 1$. The corresponding topology on $\End(\AA_\kk^n)$ is called the \emph{Zariski ind-topology}.
		\item Let $(\kk,|\cdot|)$ be a valued field with a non-trivial absolute value $|\cdot|$. The map $(x,y)\mapsto |x-y|$ is a metric on $\kk$ which yields a topology on $\kk$ in the usual way. The sets~$\mathcal{E}_{\leq d}$, being identified with a $\kk$-vector space $\kk^{N}$, where $N={\binom{d+n}{n} n}$, can be endowed with a uniform norm 
		$
		\|x_1e_1+\ldots x_Ne_N\|_{\infty}=\max_{1\leqslant i\leqslant N}|x_i|
		$
		for a given choice of basis $\langle e_1,\ldots,e_N\rangle=\mathcal{E}_{\leq d}$. Furthermore, any two such uniform norms with respect to two different bases are equivalent, and hence they define the same topology on $\mathcal{E}_{\leqslant d}$. If $(\kk,|\cdot|)$ is complete, then all norms on $\mathcal{E}_{\leqslant d}$ are equivalent to the uniform norm and therefore define the same topology. In any case, by an abuse of terminology, we call this topology \emph{Euclidean}. A set $S\subseteq\mathcal{E}$ is called \emph{Euclidean closed} if $S\cap\mathcal{E}_{\leq d}$ is closed in this topology on $\mathcal{E}_{\leq d}$ for all $d\geq 1$. The corresponding topology on $\End(\AA_\kk^n)$ is called the \emph{Euclidean ind-topology}. Note that this topology is still Hausdorff.
	\end{enumerate}
\end{definition}

\subsection{Main result}\label{sec: main result} The main result of this note is the following.

\begin{maintheorem}
	Let $\kk$ be a field, $r\geq 1$ be an integer and $w=x^r$ be a power word. Consider the corresponding power map on $\End(\AA_\kk^1)\simeq\kk[x]$, which sends every polynomial to its $r$-th iterate: 
	\[
	\wordmap\colon\kk[x]\to\kk[x],\ \ P\mapsto P^{\circ r}.
	\]
	Then the following holds.
	\begin{enumerate}
		\item If $(\kk,|\cdot|)$ is a valued field with a non-trivial absolute value, then $\wordmap$ is dominant in the Euclidean ind-topology. 
		\item The map $\wordmap$ is dominant in the Zariski ind-topology for any infinite field $\kk$. 
	\end{enumerate}
	Furthermore, the statement stays true when $w$ is replaced with any non-trivial word $w\in\mathcal{M}_N$ in the free monoid $\mathcal{M}_N$ on $N$ generators.
\end{maintheorem} 

Note that, unlike the case of dominant morphisms between algebraic varieties, the image of a dominant map in the ind-topology \emph{does not}, in general, contain a Zariski-open subset.

\begin{remark}
	\begin{itemize}
		\item[] 
		\item Whenever it is defined, the Euclidean ind-topology on $\End(\AA^n_\kk)$ is finer than the Zariski ind-topology and therefore, for $\kk$ valued, statement (1) is stronger than statement~(2).
		\item The ``Furthermore'' part of the theorem follows from its main part. Indeed, if $w=x_{i_1}^{m_1}x_{i_2}^{m_2}\ldots x_{i_s}^{m_s}\in\mathcal{M}_N$ is any word (here $m_1,\ldots,m_s>0$, as we work in the monoid), then the image of $\wordmap$ contains $\wordmap(P,\ldots,P)=P^{\circ(m_1+\cdots+m_s)}$ for all $P\in\kk[x]$, hence it is dense.
		\item Let $|\cdot|_1,\ldots,|\cdot|_n$ be pairwise non-equivalent non-trivial absolute values on $\kk$ (so in particular they induce different topologies on $\kk$). The proof of our main result and the Artin-Whaples approximation theorem \cite[Theorem 1]{ArtinWhaples} imply that, given any polynomials $Q_1,\ldots,Q_n\in\kk[x]$, $r\geqslant 1$ and $\eta>0$, there exists $P\in\kk[x]$ such that 
		\[
		\left \|P^{\circ r}-Q_i \right \|_i<\eta,
		\]
		for all $i\in\{1,\ldots,n\}$. Here, 
$P$ is a polynomial of degree at most
\[ 
d= \max\{1, \deg Q_1,\ldots,\deg Q_n\} +r
\]
(see the key proposition in Section~\ref{subsection: The key proposition}) and $\|\cdot\|_i$ denotes a norm on $\mathcal{E}_{\leq d^r}$ induced by the absolute value $|\cdot|_i$, as in Definition \ref{definition: topologies}.			
	\end{itemize}
\end{remark}

\subsection{Finitary case}

Let $\kk=\FF_q$ be a finite field. Then, the filtration \eqref{eq: filtration} in Definition~\ref{definition: topologies} is a filtration by finite sets. Fix $r\geq 2$. For each $d\geq 1$ let
	\[
	\Iterates(d,r)=\big \{P\in\mathcal{E}_{\leq d}\colon P=Q^{\circ r}\ \text{for some}\ Q\in\kk[x] \big \}.
	\]
	Notice that the ``asymptotic density'' of $r$-th iterates is zero: 
	\[
	\lim_{d\to+\infty}\frac{|\Iterates(d,r)|}{|\mathcal{E}_{\leq d}|}=0.
	\]
	Indeed, suppose that $P=Q^{\circ r}$. Then $\deg P\leq d$ if and only if $\deg Q\leq d^{1/r}$. Hence there is a surjection from $\mathcal{E}_{\leq d^{1/r}}$ to $\Iterates(d,r)$, and we deduce $|\Iterates(d,r)|\leq q^{d^{1/r}+1}$. Therefore,
	\[
	\frac{|\Iterates(d,r)|}{|\mathcal{E}_{\leq d}|}=\frac{|\Iterates(d,r)|}{q^{d+1}}\leq q^{d^{1/r}-d}\rightarrow 0\quad\textrm{when}\quad d\rightarrow+\infty.
	\]
	Nevertheless, the following questions seem interesting to us.
	\begin{question}
		Fix a finite base field $\kk=\FF_q$ and an integer $r\geq 2$. What can be said about the asymptotic of the integer sequence $|\Iterates(d,r)|$ as $d$ grows? Does there exist a limit of the rational sequence $q^{-d-1}|\Iterates(d^r,r)|$ as $d\to +\infty$?
	\end{question}
	
\subsubsection*{Acknowledgements}
We thank Serge Cantat for his helpful comments and suggestions on the draft of this paper.

\section{Power maps on  $\End(\A^1_\kk)$}

\subsection{Hasse derivative} We start with recalling some definitions.  

\begin{definition}[{\it Hasse derivative}, see e.g. {\cite[\S 1.3]{Goldschmidt}}] \label{definition: divided derivatives}
	Let $A$ be a commutative ring. For each polynomial $P(x)=p_nx^n+p_{n-1}x^{n-1}+\dots+ p_1x+p_0 \in A[x]$ and non-negative integer $j\leq n$, define the \emph{$j$-th Hasse derivative} of $P$ by
	\[ 
	P^{[j]}(x) =\sum\limits_{k=j}^n p_{k} \binom{k}{j} x^{k-j}. 
	\]
\end{definition}

Note that the $j$-th Hasse derivative of $P$ satisfies $P^{(j)} = j! P^{[j]}$, where $P^{(j)}$ denotes the usual $j$-th derivative. In what follows, we will use the following two properties of Hasse derivatives.

\begin{proposition}\label{prop: Hasse derivative}
	Let $A$ be a commutative ring, and $P,Q\in A[x]$. Then the following holds.
	\begin{enumerate}
		\item {\bf Taylor's formula:}
		\[ 
		P(a+b) = \sum\limits_{j \ge 0} P^{[j]}(a) \, b^j 
		\]
		for any $a,b\in A$.
		\item {\bf Leibniz rule:}
		\[ 
		(PQ)^{[j]} = \sum\limits_{\ell=0}^j P^{[\ell ]} \, Q^{[j -\ell]}.
		\]
	\end{enumerate}
\end{proposition}

\begin{remark}\label{rem: Hasse scaling variable}
	For any $a\in\kk$ and $T(x)=\sum_{i=0}^{m}t_ix^i\in\kk[x]$, we have
	\[
	T(ax)^{[j]}=\sum_{i=j}^{m}t_ia^i\binom{i}{j}x^{i-j}=a^jT^{[j]}(ax).
	\]
\end{remark}

In the sequel we will apply Definition~\ref{definition: divided derivatives} with the ring $A=K [\varepsilon, \varepsilon^{-1} ]$ of Laurent polynomials in $\varepsilon$.

\begin{definition}
	Let $A=K [\varepsilon, \varepsilon^{-1} ]$ and let $P,Q\in A[x]$ be any two polynomials. Fix an integer $\ell\in\ZZ$. We say that \emph{$P$ and $Q$ are equivalent modulo $\varepsilon^\ell$}, and write $P\equiv Q\mod\varepsilon^\ell$, if 
	\[
	P-Q\in\varepsilon^\ell\kk[\varepsilon][x].
	\]
\end{definition}
\begin{remark}
	This is indeed an equivalence relation on the ring $A[x]$. When $(\kk,|\cdot|)$ is a valued field with a non-trivial absolute value, one has $P\equiv Q\mod\varepsilon^\ell$ if and only if $P-Q=O(\varepsilon^\ell)$ when $\varepsilon\to 0$. If no confusion arises, we will use both notations below.
\end{remark}

\subsection{The key proposition} \label{subsection: The key proposition}

The proof of our main result relies on the following.

\begin{keyproposition} 
Let $r \ge 2$ be an integer and $\kk$ be a field with at least $r$ elements. Let $n \ge 2$ be an integer and $Q \in \kk [x]$ be a polynomial of degree at most $n-1$. Then there exists a family $\varepsilon \mapsto P_{\varepsilon}$ of polynomials of degree at most $n+r-1$  parametrized by $\A^1_{\kk} \setminus \{ 0 \}$, such that the family $\varepsilon \mapsto (P_{\varepsilon})^{\circ r}$ extends to a family parametrized by $\A^1_k$ and whose value at $\varepsilon =0$ is  $Q$.
\end{keyproposition}

The rest of this section is devoted to the proof of this statement. Let $a_1,\ldots,a_{r-1}$ be distinct elements of $\kk^\times$. Put $a_r=0$.

\begin{lemma}\label{lem: polynomial L}
	There exists a unique polynomial $L\in \kk[x]$ of degree $\deg L\leq n+r-2$ such that 
	\[
	L(0)=a_1,\ \ L(a_k)=a_{k+1}\ \ \text{for all}\ k\in\{1,\ldots,r-1\},\ \ \text{and}\ L^{[j]}(0)=0\ \ \text{for all}\ j\in\{1,\ldots,n-1\}.
	\]
\end{lemma}
\begin{proof}
	Indeed, let $L(x)=\ell_{n+r-2}x^{n+r-2}+\dots+\ell_1x+\ell_0\in \kk[x]$. The condition $L^{[j]}(0)=0$ for all $1\leq j\leq n-1$ implies that $\ell_1=\ell_2=\dots=\ell_{n-1}=0$, hence $L$ is of the form
	\[
	L(x)=\ell_{n+r-2}x^{n+r-2}+\dots +\ell_n x^n+a_1.
	\]
	The conditions  $L(a_k)=a_{k+1}$ for all $1 \leq k\leq r-1$ then give a system of $r-1$ linear equations in $r-1$ variables $\ell_{n+r-2},\ldots,\ell_{n}$. The matrix of this system is the Vandermonde matrix ${\mathrm V}(a_1,\ldots,a_{r-1})$. Since $a_i$'s are assumed pairwise distinct, the Vandermonde determinant is not zero; this achieves the proof.
\end{proof}

\begin{remark}
	Another way to prove Lemma \ref{lem: polynomial L} is to use the Lagrange interpolation formula:
	\[
	L(x)=a_1+\sum_{k=1}^{r-1}(a_{k+1}-a_1)\frac{x^n}{a_k^n}\prod_{l\neq k}\frac{x-a_l}{a_k-a_l}.
	\]
	This polynomial $L$ coincides with the one obtained in Lemma \ref{lem: polynomial L}.
\end{remark}

Now, define the polynomial  

\begin{equation}\label{eq: polynomial R}
	R(x)=\prod_{k=1}^{r-1}(a_k-x)\in\kk[x]
\end{equation}

and set 
\begin{equation}\label{eq: c-constant}
c=R(0)^{n+1}\prod_{l=1}^{r-1}R'(a_l).
\end{equation}

Define the polynomial 
\begin{equation}\label{eq: polynomial P}
	P(x)=\varepsilon^{(r-1)(2n-3)}R(\varepsilon^{2r}x)(\varepsilon^rx^n+c^{-1}Q)+\varepsilon^{-2r}L(\varepsilon^{2r}x)\in\kk[\varepsilon,\varepsilon^{-1}][x]. 
\end{equation}

The key proposition is a direct consequence of the following statement. 

\begin{proposition} \label{prop: main proposition}
One has $P^{\circ r}\equiv Q\mod \varepsilon$.
\end{proposition}

To prove it, we will need several lemmas.

\begin{lemma}\label{lem: base of induction}
	For every $k\in \{1,\ldots,r-1\}$, one has 
	\[
	P'(\varepsilon^{-2r}a_k)\equiv\varepsilon^{3-2n}R'(a_k)a_k^n\mod\varepsilon^{4-2n}.
	\]
\end{lemma}
\begin{proof}
	First, we differentiate $P$ and obtain
	\[
	P'(x)=\varepsilon^{(r-1)(2n-3)}\Big [R'(\varepsilon^{2r}x)\varepsilon^{2r}(\varepsilon^rx^n+c^{-1}Q)+R(\varepsilon^{2r}x)(n\varepsilon^{r}x^{n-1}+c^{-1}Q')  \Big]+L'(\varepsilon^{2r}x).
	\]
	Notice that $R(a_k)=0$ and $L'(a_k)\equiv 0\mod{\varepsilon^{4-2n}}$ since $4-2n\leq 0$ and $L'(a_k)\in\kk$. Hence
	\[
	P'(\varepsilon^{-2r}a_k)\equiv \varepsilon^{3-2n} R'(a_k)a_k^n+\varepsilon^{(r-1)(2n-3)+2r} R'(a_k)c^{-1}Q(\varepsilon^{-2r}a_k)\ \mod{\varepsilon^{4-2n}},
	\]
	and it suffices to show that
	\[
	\varepsilon^{(r-1)(2n-3)+2r}Q(\varepsilon^{-2r}a_k)\equiv 0\mod\varepsilon^{4-2n}.
	\]
	This is equivalent to saying that $\varepsilon^{(r-1)(2n-3)+2r-4+2n}Q(\varepsilon^{-2r}a_k)\in\kk[\varepsilon]$. But, writing $Q=q_{n-1}x^{n-1}+ \dots +q_0$, this Laurent polynomial equals
	\[
	\varepsilon^{2nr-r-1}\sum_{i=1}^{n}q_{n-i}(\varepsilon^{-2r}a_k)^{n-i}=\sum_{i=1}^{n}q_{n-i}a_k^{n-i}\varepsilon^{2ri-r-1}
	\]
	and we observe that $2ri-r-1\geq 0$ for all $i\geq 1$, which finishes the proof.
\end{proof}

\begin{lemma}\label{lem: Pj}
	For all $j\geq 1$, one has
\[ P^{[j]}(\varepsilon^{-2r}a_k)=O\left  (\varepsilon^{2r(j-1)-2n+3} \right ). \]
\end{lemma}
\begin{proof}
By Remark \ref{rem: Hasse scaling variable}, we have
\[
\left ( \varepsilon^{-2r}L(\varepsilon^{2r}x) \right )^{[j]}=\varepsilon^{2r(j-1)}L^{[j]}(\varepsilon^{2r}x),\ \ \text{hence}\ \ \left ( \varepsilon^{-2r}L(\varepsilon^{2r}x) \right )^{[j]}(\varepsilon^{-2r}a_k)=O\left (\varepsilon^{2r(j-1)}\right ).
\]
Next, we will show that if $u\geq 0$ and $\deg S\leq v$, then for any $a\in\kk$ one has
\begin{equation}\label{eq: Lemma on P^j}
\left (T(\varepsilon^u x)S \right )^{[j]}\left (\varepsilon^{-u}a \right )=O\left (\varepsilon^{(j-v)u}\right ).
\end{equation}
By Leibniz rule, we get
\[
\left (T(\varepsilon^ux) S\right )^{[j]}=\sum_{\ell=0}^{j}\left( T(\varepsilon^ux)\right )^{[\ell]}S^{[j-\ell]}=\sum_{\ell=0}^{j}\varepsilon^{u\ell} T^{[\ell]}(\varepsilon^ux)S^{[j-\ell]}.
\]
Now, as $\deg S^{[j-\ell]}\leq\deg S-(j-\ell)\leq v-j+\ell$ and $u\geq 0$, we have
\[
\varepsilon^{u\ell} T^{[\ell]}(\varepsilon^ux)S^{[j-\ell]}\Big|_{x=\varepsilon^{-u}a}=\varepsilon^{u\ell} O\left (\varepsilon^{-u(v-j+\ell)} \right )=O\left (\varepsilon^{(j-v)u}\right ),
\] 
as claimed.

We now apply \eqref{eq: Lemma on P^j} to the first summand of \eqref{eq: polynomial P}. Namely, taking $u=2r$, $T=R$ and $S=x^n$ we get
\[
\left (R(\varepsilon^{2r}x)\varepsilon^rx^n \right )^{[j]}\Big|_{x=\varepsilon^{-2r}a_k}= O\left (\varepsilon^{2r(j-n)+r} \right ),
\]
while taking $u=2r$, $T=R$ and $S=c^{-1}Q$, a polynomial of degree $\leq n-1$, we get
\[
\left (R(\varepsilon^{2r}x)c^{-1}Q\right )^{[j]}\Big|_{x=\varepsilon^{-2r}a_k}=O\left (\varepsilon^{2r(j-n)+2r} \right ).
\]
Further multiplication by $\varepsilon^{(r-1)(2n-3)}$ gives $O\left (\varepsilon^{2r(j-1)-2n+3} \right )$ and $O\left (\varepsilon^{2r(j-1)+r-2n+3} \right )$, respectively. Since $2r(j-1)-2n+3<\min\{2r(j-1)+r-2n+3,2r(j-1)\}$, we are done.
\end{proof}

\begin{lemma}\label{lem: main lemma base of induction}
	One has
	\[
	P\equiv \varepsilon^{-2r}a_1+\varepsilon^{(r-1)(2n-3)}Qc^{-1}R(0)\mod\varepsilon^{(r-1)(2n-3)+1}.
	\]
\end{lemma}
\begin{proof}
	We look at each of the two summands in \eqref{eq: polynomial P}. First, we notice that 
	\[
	\varepsilon^{-2r}L(\varepsilon^{2r}x)\equiv \varepsilon^{-2r}a_1\mod\varepsilon^{(r-1)(2n-3)+1}.
	\]
	Indeed, we can write $L(x)=\sum_{i=0}^{r-2}\ell_{n+i}x^{n+i}+a_1$, as in the proof of Lemma \ref{lem: polynomial L}. Therefore, we just need to verify that
	\[
	\sum_{i=0}^{r-2}\ell_{n+i}x^{n+i}\varepsilon^{2r(n+i)-2r-(r-1)(2n-3)-1}
	\]
	is a polynomial in $\kk[x,\varepsilon]$, i.e. has non-negative powers in $\varepsilon$. But for all $i\in\{0,\ldots,r-2\}$ we have
	\[
	2r(n+i)-2r-(r-1)(2n-3)-1=2ri+r+(2n-4)\geq r\geq 0,
	\]
	since $2n-4\geq 0$ and $i\geq 0$.
	
	Second, we claim that 
	\[
	\varepsilon^{(r-1)(2n-3)}Qc^{-1}R(0)\equiv \varepsilon^{(r-1)(2n-3)}R(\varepsilon^{2r}x)(\varepsilon^rx^n+Qc^{-1})\mod \varepsilon^{(r-1)(2n-3)+1}.
	\]
	This is equivalent to showing
	\[
	Qc^{-1}R(0)\equiv R(\varepsilon^{2r}x)(\varepsilon^rx^n+Qc^{-1})\mod \varepsilon,
	\]
	which is obvious, as the constant term of $Qc^{-1}R(0)-R(\varepsilon^{2r}x)Qc^{-1}$, viewed as polynomial in $\varepsilon$, equals zero.
\end{proof}

\begin{lemma}
	For all $k\in\{1,\ldots, r\}$, one has
	\begin{equation}\label{eq: iterations}
	P^{\circ k}\equiv\varepsilon^{-2r}a_k+\varepsilon^{(r-k)(2n-3)}Qc^{-1}R(0)\prod_{l=1}^{k-1}R'(a_l)a_l^n\mod\varepsilon^{(r-k)(2n-3)+1}.
	\end{equation}
\end{lemma}
\begin{proof}
We proceed by induction on $k$. The base $k=1$ is ensured by Lemma \ref{lem: main lemma base of induction}. So, we assume that the statement holds for $k$ and will show it for $k+1$. By the induction hypothesis, we can write
	\[
	P^{\circ k}(x)=\varepsilon^{-2r}a_k+\varepsilon^{(r-k)(2n-3)}\left [Q(x)c^{-1}R(0)\prod_{l=1}^{k-1}R'(a_l)a_l^n+\varepsilon T(\varepsilon,x) \right ]
	\]
	with $T(\varepsilon,x)\in\kk[\varepsilon,x]$. Let $S(\varepsilon,x)=Q(x)c^{-1}R(0)\prod_{l=1}^{k-1}R'(a_l)a_l^n+\varepsilon T(\varepsilon,x)$, then
	\begin{equation}\label{eq: Decomposition of Pk}
	P^{\circ k}(x)=\varepsilon^{-2r}a_k+\varepsilon^{(r-k)(2n-3)}S(\varepsilon,x),
	\end{equation}
	where 
	\[
	S\equiv Q(x)c^{-1}R(0)\prod_{l=1}^{k-1}R'(a_l)a_l^n\mod\varepsilon.
	\]
	By Taylor's formula from Proposition \ref{prop: Hasse derivative} applied to \eqref{eq: Decomposition of Pk}, we get
	\begin{equation}\label{eq: P^k+1}
	P^{\circ (k+1)}(x)=P(\varepsilon^{-2r}a_k+\varepsilon^{(r-k)(2n-3)}S)=\sum_{j\geq 0}^{}P^{[j]}(\varepsilon^{-2r}a_k)\varepsilon^{j(r-k)(2n-3)}S^j.
	\end{equation}
	By Lemma \ref{lem: Pj}, we have $P^{[j]}(\varepsilon^{-2r}a_k)=O(\varepsilon^{2r(j-1)-2n+3})$, while $S^j\equiv U(x)\mod\varepsilon$ for some $U(x)\in\kk[x]$. Therefore,
	\[
	P^{[j]}(\varepsilon^{-2r}a_k)\varepsilon^{j(r-k)(2n-3)}S^j=O\left (\varepsilon^{2r(j-1)-2n+3+j(r-k)(2n-3)} \right ).
	\] 
	Consider the function
	\[
	h\colon \NN\to\ZZ,\ \ j\mapsto 2r(j-1)-2n+3+j(r-k)(2n-3). 
	\]
	Note that $h$ is a strictly increasing affine function, since $2r+(r-k)(2n-3)>0$. Thus $h(1)+1\leq h(j)$ for all $j\geq 2$, which gives
	\[
	(r-k)(2n-3)-2n+4=(r-(k+1))(2n-3)+1\leq h(j).
	\]
	We conclude that 
	\[
	\sum_{j\geq 2}^{}P^{[j]}(\varepsilon^{-2r}a_k)\varepsilon^{j(r-k)(2n-3)}S^j=O\left(\varepsilon^{(r-(k+1))(2n-3)+1} \right),
	\]
	and thus formula \eqref{eq: P^k+1} implies
	\begin{equation}\label{eq: two terms}
	P^{\circ(k+1)}(x)\equiv P(\varepsilon^{-2r}a_k)+P^{[1]}(\varepsilon^{-2r}a_k)\varepsilon^{(r-k)(2n-3)}S\mod \varepsilon^{(r-(k+1))(2n-3)+1}.
	\end{equation}
	To achieve the proof, it remains to show that the right hand side of \eqref{eq: two terms} is equivalent, modulo $\varepsilon^{(r-(k+1))(2n-3)+1}$, to
	\[
	\varepsilon^{-2r}a_{k+1}+\varepsilon^{(r-(k+1))(2n-3)}Qc^{-1}R(0)\prod_{l=1}^{k}R'(a_l)a_l^n.
	\]
	But $P(\varepsilon^{-2r}a_k)=\varepsilon^{-2r}a_{k+1}$ by Lemma \ref{lem: polynomial L}. On the other hand,
	\[
	P^{[1]}(\varepsilon^{-2r}a_k)\varepsilon^{(r-k)(2n-3)}S\equiv \varepsilon^{(r-(k+1))(2n-3)}Qc^{-1}R(0)\prod_{l=1}^{k}R'(a_l)a_l^n \mod \varepsilon^{(r-(k+1))(2n-3)+1}.
	\]
	Indeed, this is equivalent to
	\[
	P^{[1]}(\varepsilon^{-2r}a_k)S\equiv \varepsilon^{3-2n}Qc^{-1}R(0)\prod_{l=1}^{k}R'(a_l)a_l^n \mod \varepsilon^{4-2n},
	\]
	which immediately follows from Lemma \ref{lem: base of induction} and the definition of $S$.
\end{proof}

We are now in position to prove Proposition \ref{prop: main proposition}.

\begin{proof}[Proof of Proposition \ref{prop: main proposition}]
	Put $k=r$ in the formula \eqref{eq: iterations}. Recall that $a_r=0$, hence
	\[
	P^{\circ r}\equiv Qc^{-1}R(0)\prod_{l=1}^{r-1}R'(a_l)a_l^n\mod\varepsilon.
	\]
	It remains to notice that, by the definitions \eqref{eq: polynomial R} and \eqref{eq: c-constant} of $R$ and $c$, one has
	\[
	c^{-1}R(0)\prod_{l=1}^{r-1}R'(a_l)a_l^n=\frac{1}{R(0)^n}\prod_{l=1}^{r-1}a_l^n=1.  \qedhere
	\]
\end{proof}

\subsection{Proof of the main result}

Let $\wordmap\colon\End(\AA_\kk^1)\to\End(\AA_\kk^1)$ be the $r$-th iterate map defined by the word $w=x^r$, and let $\mathcal{W}=\wordmap(\End(\AA_\kk^1))$ be its image. We will show that $\mathcal{W}$ is dense in $\End(\AA_\kk^1)$ in the Zariski and Euclidean ind-topologies. So, let $Q\in\End(\AA_\kk^1)$ be any polynomial and $\mathcal{U}\ni Q$ be any of its open neighbourhoods. We then need to show that $\mathcal{U}\cap\mathcal{W}\ne\varnothing$. Put $d=\max\{1,\deg Q\}+r$. Then it is sufficient to show that $\mathcal{U}\cap\mathcal{E}_{\leq d^r}$ contains an element of $\mathcal{W}$. In what follows, $\mathcal{E}_{\leq d^r}$ is identified with the vector space $\kk^N$, where $N=d^r+1$. By Proposition \ref{prop: main proposition}, there exists $P_{\varepsilon}\in\kk[\varepsilon,\varepsilon^{-1}][x]$ such that $Q=P_{\varepsilon}^{\circ r}+\varepsilon T(\varepsilon,x)$, where $T(\varepsilon,x)\in\kk[\varepsilon,x]$.
	
	Suppose that $(\kk,|\cdot|)$ is a valued field with a non-trivial absolute value. Let us first notice that $\kk$ contains elements which are arbitrarily close to zero. More precisely, $|\kk^\times|=\{|\varepsilon|\colon\varepsilon\in\kk^\times \}$ is a subgroup of $\RR_{>0}$, which is either dense (if $1$ is its accumulation point), or is of the form $\{t^\ZZ\}$ for some $t\in (0,1)$. Now, the set $\mathcal{U}\cap\mathcal{E}_{\leq d^r}$ is open in $\mathcal{E}_{\leq d^r}$, hence there are $\delta\in\RR_{>0}$ and an open $\delta$-ball $B_\delta(Q)\subset \mathcal{U}\cap\mathcal{E}_{\leq d^r}$ centred at $Q$. Then, by choosing $\varepsilon\in\kk^{\times}$ with $|\varepsilon|\cdot\|T\|<\delta$, we obtain $\|Q-P_\varepsilon^{\circ r}\|=|\varepsilon|\cdot\|T\|<\delta$, hence $P_{\varepsilon}^{\circ r}\in B_\delta(Q)\subset\mathcal{U}\cap \mathcal{E}_{\leq d^r}\cap\mathcal{W}$, as required.
	
	In the case of an arbitrary infinite base field $\kk$, we proceed as follows. Writing $\mathcal{E}_{\leq d^r}\setminus (\mathcal{U}\cap\mathcal{E}_{\leq d^r})$ as the zero set of the ideal $(F_1,\ldots,F_m)$, we notice that $F_i(Q)\ne 0$ for some $i\in\{1,\ldots,m\}$; here we identify $Q$ with a point in $\kk^N$. Consider $G(\varepsilon)=F_i( P_\varepsilon^{\circ r})=F_i(Q-\varepsilon T)$ as a polynomial of $\varepsilon$. Note that $G$ is not a zero polynomial, as $G(0)=F_i(Q)\ne 0$. Since $\kk$ is infinite, there exists $\varepsilon\in\kk^\times$ such that $G(\varepsilon)=F_i(P_\varepsilon^{\circ r})\ne 0$, thus $P_{\varepsilon}^{\circ r}\in \mathcal{U}\cap\mathcal{E}_{\leq d^r}$. This achieves the proof.

\def\bibindent{2.5em}

\bibliographystyle{alphadin}
\bibliography{biblio}

\address{IMB, Universit\'{e} de Bordeaux, 351 Cours de la Lib\'{e}ration, 33405 Talence Cedex, France}
\email{pascal.autissier@math.u-bordeaux.fr}
\email{jean-philippe.furter@math.u-bordeaux.fr}
\email{egor.yasinsky@u-bordeaux.fr}

\end{document}